\documentclass[10pt]{article}
\usepackage{amssymb}
\usepackage{amsfonts}
\usepackage{graphicx}
\usepackage[center]{caption}
\usepackage{color}
\usepackage{amsmath}
\usepackage{enumerate}

\newtheorem{theorem}{Theorem}[section]

\newtheorem{assumption}[theorem]{Assumption}

\newtheorem{case}[theorem]{Case}
\newtheorem{claim}[theorem]{Claim}

\newtheorem{corollary}[theorem]{Corollary}

\newtheorem{definition}[theorem]{Definition}

\newtheorem{lemma}[theorem]{Lemma}

\newtheorem{remark}[theorem]{Remark}

\newenvironment{proof}[1][Proof]{\noindent\textit{#1.} }
{\hfill \rule{0.5em}{0.5em}}

\newcommand{\R}{\mathbb{R}}

\begin{document}

\title{\textbf{Existence and convergence to a propagating terrace
in one-dimensional reaction-diffusion equations}}

\author{Arnaud Ducrot$^{\hbox{ \small{a}}}$, Thomas Giletti$^{\hbox
{ \small{b}}}$, Hiroshi Matano$^{\hbox{ \small{c}}}$ \\
\footnotesize{$^{\hbox{a }}$UMR CNRS 5251, Universit\'e de Bordeaux, 33000 Bordeaux, France}\\
\footnotesize{$^{\hbox{b }}$UMR 6632 LATP, Universit\'e Aix-Marseille,  Facult\'e des Sciences et Techniques, 13397 Marseille, France}\\
\footnotesize{$^{\hbox{c }}$Graduate School of Mathematical
Sciences, University of Tokyo, Komaba, Tokyo 153-8914, Japan}
}

\maketitle

\begin{abstract}
We consider one-dimensional reaction-diffusion equations for a large
class of spatially periodic nonlinearities -- including multistable
ones -- and study the asymptotic behavior
of solutions with Heaviside type initial data. Our analysis reveals
some new dynamics where the profile of the propagation is not
characterized by a single front, but by a layer of several fronts
which we call a terrace. Existence and convergence to such a terrace
is proven by using an intersection number argument, without much
relying on standard linear analysis. Hence, on top of the peculiar
phenomenon of propagation that our work highlights, several
corollaries will follow on the existence and convergence to
pulsating traveling fronts even for highly degenerate
nonlinearities that have not been treated before.

\vspace{0.2in}\noindent \textbf{Key words}.
Multistable reaction-diffusion equation,
periodic environment,
long time behavior,
propagating terrace,
zero-number argument.

\vspace{0.1in}\noindent
\textbf{2010 Mathematical Subject Classification}.\
35K55, 35C07, 35B08, 35B40
\end{abstract}

\section{Introduction}

We consider in this work a Cauchy problem for the following
reaction-diffusion equation in one space dimension:
\begin{equation}\tag{$E$}\label{eqn1}
\partial_t u(t,x)= \partial_{xx} u(t,x)+ f(x,u(t,x)), \ \
\forall (t,x) \in [0,+\infty) \times \R,
\end{equation}
supplemented with the initial condition
\begin{equation}\label{initial}
u(0,x)=u_0(x)\geq 0, \ \ \forall x\in\R.
\end{equation}
Here the function $f \in C^1 (\R^2;\R)$ satisfies the periodicity
condition
\begin{equation}\label{eqn:f}
f(x+L,u) \equiv f(x,u) \ \mbox{ and } f(x,0) \equiv 0,
\end{equation}
for some $L>0$. We will assume, throughout this paper, that there
exists a positive and $L$-periodic stationary solution $p(x)$ of
\eqref{eqn1}:
\begin{equation}\label{P}\tag{$P$}
\left\{
\begin{array}{l}
p''(x)+f(x,p(x))=0, \ \ \forall x\in\R,\vspace{3pt}\\
p(x)>0,\ \ p(x+L)\equiv p(x).
\end{array}\right.
\end{equation}
The function $p$ is also a stationary solution of the following
auxiliary equation, the $L$-periodic counterpart of \eqref{eqn1}:
\begin{equation}\tag{$E_{per}$}\label{eqn1-per}
\left\{
\begin{array}{l}
\partial_t u(t,x)= \partial_{xx} u(t,x)+ f(x,u(t,x)), \ \ \forall
(t,x) \in [0, +\infty) \times \R, \vspace{3pt}\\
u (t, \cdot) \ L\mbox{-periodic for any } t \in \R.
\end{array}
\right.
\end{equation}
It is obvious that any solution of \eqref{eqn1-per} is also a
solution of \eqref{eqn1}. Equation \eqref{eqn1-per} will later
play an important auxiliary role in the analysis of \eqref{eqn1}.\\

Our aim is to investigate the profile of solutions of \eqref{eqn1}
connecting the two stationary states $0$ and $p$. In particular,
we will study the long time behavior of solutions with Heaviside-type
initial data. Roughly speaking, our main results state that the
solution will converge to what we call a \lq\lq propagating
terrace\rq\rq, the meaning of which will be specified later.
This result, in particular, implies the existence of pulsating
traveling waves (or a set of traveling waves) under rather mild
assumptions.\\

We now state our main assumptions. The first one is concerned with
the attractiveness of $p$ with respect to at least one compactly
supported initial data. Our main theorems (Theorems~\ref{heavip2}
and~\ref{heavip2-CV}) will only need this assumption:

\begin{assumption}\label{assumption-p1'}
There exists a solution $u$ of \eqref{eqn1}-\eqref{initial} with
compactly supported initial data $0 \leq u_0 (x) < p(x)$ that
converges locally uniformly to $p$ as $t \to +\infty$.
\end{assumption}

This assumption covers a wide variety of nonlinearities that
include not only such standard ones as monostable, bistable or
combustion nonlinearities, but also much more general and complex
ones. For instance, it even allows an infinite number of
stationary solutions between 0 and $p$. In this paper, we will
show that this rather weak condition is in fact sufficient for
deriving our main results on the convergence to a
\lq\lq propagating terrace\rq\rq.

The next assumption guarantees that our propagating terrace
consists of a single (pulsating) traveling wave.  Thus, under
this additional assumption, our main results imply the existence
of a pulsating traveling wave, as well as the convergence
of solutions to this traveling wave (Theorem~\ref{heavip1}):

\begin{assumption}\label{assumption-p2'}
There exists no $L$-periodic stationary solution $q$ with
$0< q(x)<p(x)$ that is both isolated from below and stable from
below with respect to \eqref{eqn1-per}.
\end{assumption}

Let us clarify the notions introduced in this assumption. A
stationary solution $q$ of \eqref{eqn1-per} is said to be
{\bf isolated from below} (resp. {\bf above}) if there exists
no sequence of other stationary solutions converging to $q$
from below (resp. above). A stationary solution $q$ is said to
be {\bf stable from below} (resp. {\bf above}) with respect
to equation \eqref{eqn1-per} if it is stable in the $L^\infty$
topology under nonpositive (resp. nonnegative) perturbations.
Otherwise, $q$ is called {\bf unstable from below} (resp.
{\bf above}). It is known that, if $q$ is isolated from below,
then it is stable from below if and only
if there exists a solution $u<q$ converging to $q$ as
$t \rightarrow +\infty$, and unstable from below if and only if
there exists an ancient solution (that is, a solution defined
for all sufficiently negative~$t$) $u<q$ converging to $q$ as
$t \rightarrow -\infty$ (see Theorem~8 in~\cite{Matano84}).

Note that this additional assumption holds for a large class of
standard nonlinearities including the following:

\begin{case}[Monostable nonlinearity]
There exists no $L$-periodic stationary solution $q$ satisfying
$0<q(x)<p(x)$ for all $x\in \R$. Furthermore, $0$ is unstable
from above.
\end{case}

\begin{case}[Bistable nonlinearity]
The stationary solution $0$ is stable from above with respect
to \eqref{eqn1-per}, and $p$ is stable from below with respect to
\eqref{eqn1-per}. Furthermore, all other stationary solutions
between 0 and $p$ are unstable.
\end{case}

\begin{case}[Combustion nonlinearity]
There exists a family of $L$-periodic stationary solutions
$(q_\lambda)_{\lambda \in [0,1]}$ that forms a continuum in
$L^\infty(\R)$ and satisfies $0=q_0<q_1<p$.
Furthermore, there exists no stationary solution $q$ satisfying
$q_1  (x) < q(x) < p(x)$ for all $x \in \R$.
\end{case}

A classical example of the bistable nonlinearity is the
Allen-Cahn nonlinearity $u(1-u)(u-a(x))$, where $0<a(x)<1$,
$a(x+L)\equiv a(x)$. An important subclass of the monostable
nonlinearity is the KPP type nonlinearity, in which $0$ is
assumed to be linearly unstable and $f$ is sublinear with respect
to $u$; a typical example being $(R(x)-u)u$, with
$R(x+L)\equiv R(x)>0$.

KPP type equations have been widely studied, even in the periodic
setting, by numerous authors including \cite{BMR08,BMR11,Hamel08,
HamRoq11,Wein02}. While most of those studies rely heavily on
the linear instability of $0$, our approach in the present paper
largely avoid the need for linear analysis, allowing our results
to be applicable even to strongly degenerate situations that
have not been treated before.


We now introduce some notions which will play a fundamental
role in this paper. We begin with the following:

\begin{definition}\label{def:steep}
Let $u_1,\,u_2$ be two entire solutions of \eqref{eqn1}. We say
that $u_1$ is \textbf{steeper than} $u_2$ if for any $t_1$, $t_2$
and $x_1$ in $\R$ such that $u_1 (t_1 ,x_1 ) = u_2 (t_2,x_1)$,
we have either
\[
u_1 (\cdot + t_1,\cdot) \equiv u_2 (\cdot + t_2, \cdot) \ \
\hbox{or} \ \
\partial_x  u_{1} (t_1,x_1) < \partial_x u_{2} (t_2,x_1).
\]
\end{definition}

Here, by an ``entire solution" we mean a solution that is defined
for all $t\in\R$.
The above property implies that the graph of the solution
$u_1$ (at any chosen time moment $t_1$) and that of the solution
$u_2$ (at any chosen time moment $t_2$) can intersect at most
once unless they are identical, and that if they intersect at a
single point, then $u_1-u_2$ is positive on the left-hand side of
the intersection point, while negative on the right-hand side.
Note that, according to this definition, if the ranges of $u_1$
and $u_2$ are disjoint, then $u_1$ and $u_2$ are steeper than
each other, since their graphs never intersect.

\begin{definition}[Pulsating traveling wave]\label{def:puls}
Given two distinct periodic stationary states $p_1$ and $p_2$, by a
\textbf{pulsating traveling wave solution} (or \textbf{pulsating
traveling front}) of \eqref{eqn1} connecting $p_1$ to $p_2$, we
mean any entire solution $u$ satisfying, for some $T >0$,
$$
u(t,x-L)=u(t+T,x),
$$
for any $x\in \R$ and $t \in \R$, along with the asymptotics
$$
u(-\infty,\cdot)=p_1 (\cdot) \; \mbox{ and } \; u(+\infty,\cdot)
=p_2 (\cdot),
$$
where the convergence is understood to hold locally uniformly
in the space variable. The ratio $c:=\frac{L}{T} >0$ is called
the \textbf{average speed} (or simply the~\textbf{speed}) of
this pulsating traveling wave.
\end{definition}

\begin{remark}
One can easily check that, for any $c>0$, $u(t,x)$ is a pulsating
traveling wave connecting $p_1$ to $p_2$ with speed $c$ if and only
if it can be written in the form $u(t,x) = U (x-ct,x)$, where
$U(z,x)$ satisfies
\begin{equation*}
\begin{array}{c}
 U (\cdot,x+L ) \equiv U (\cdot,x), \vspace{3.5pt}\\
 U (+\infty, \cdot) = p_1 (\cdot) \ \mbox{ and }  U (-\infty,\cdot)
 =p_2(\cdot),
\end{array}
\end{equation*}
along with the following equation that is
equivalent to \eqref{eqn1}:
\[
(\partial_x+\partial_z)^2U+cU_z+f(x,U)=0,\ \ \forall (z,x)\in\R^2.
\]
\vspace{-15pt}
\end{remark}

Let us recall some known results on traveling waves from the
literature.  In the case of spatially homogeneous problems,
existence of traveling waves is well studied (see for instance
\cite{BN92} for a review of the area).  More precisely, in the
KPP case, there exists a continuum of admissible speeds
$[c^*, +\infty)$, while in the bistable or combustion cases,
the admissible speed is unique. Stability and convergence to
those traveling waves are also studied extensively. Among
other things, in the one-dimensional KPP case
(or, more generally, the monostable case), Uchiyama~\cite{Uchi78},
Bramson~\cite{Bramson83} and Lau~\cite{Lau85} proved that
solutions of the Cauchy problem with compactly supported initial
data converge to the traveling front with minimal speed as
$t\to\infty$.
  In this case, the solution does not converge to
the traveling wave with an asymptotic phase, but a phase drift of
order~$\ln t$ occurs~\cite{Bramson83}. Similar results hold for
multi-dimensional problems as long as the parameters of the
equation are invariant in the direction of propagation~\cite{MR??}.

In the case of spatially periodic problems, the state of research
is slightly behind, for obvious technical difficulties. Nonetheless,
in the KPP case, the existence of a continuum of admissible speeds
is well established, as in the case of spatially homogeneous
problems.  It is also known that there is a close relation between
the speed of a traveling wave $u(x,t)$ and its decay rate as
$x \rightarrow +\infty$, at least under some assumptions on the
linearized problem around~$0$; the smaller the speed~$c$, the
faster the decay; hence steeper the front profile.
Convergence to those traveling waves was studied in
\cite{BMR08,BMR11,HamRoq11} in a periodic framework. More precisely,
it has been shown that if the initial data has the same exponential
decay as a given traveling wave as $x \rightarrow+\infty$, then
the solution of the Cauchy problem converges to this traveling wave
as $t \rightarrow +\infty$. However, the case of very fast decaying
initial data (for instance, a Heaviside or a compactly supported
function) has been left open up to now in the periodic framework,
although the appearance of some phase drift of order~$\ln t$ has
also been highlighted in~\cite{HNRR}.

\subsection{The notion of terrace}

Let us now come back to the main theme of the present paper ---
a propagating terrace.  As we mentioned earlier, a traveling wave
is a special case of a propagating terrace, but the latter is
a more suitable notion for describing typical frontal behaviors
in equations of multistable nature. The aim of the present
paper is to study properties of propagating terraces in a spatially
periodic setting, thereby generalizing (and improving) some of
the aforementioned results on pulsating traveling waves.

\begin{definition}\label{def:terrace}
A \textbf{propagating terrace} connecting $0$ to $p$ is a pair
of finite sequences
$(p_k)_{0 \leq k \leq N}$ and $(U_k)_{1 \leq k \leq N}$
such that:
\begin{itemize}
\item Each $p_k$ is an $L$-periodic stationary solution of
\eqref{eqn1} satisfying
\[
p=p_0 >p_1 >...>p_N =0.
\]
\item For each $1 \leq k \leq N$, $U_k$ is a pulsating traveling
wave solution of \eqref{eqn1} connecting $p_k$ to $p_{k-1}$.
\item The speed $c_k$ of each $U_k$ satisfies
$0<c_1\leq c_2\leq\cdots\leq c_N$.
\end{itemize}
Furthermore, a propagating terrace
$T=((p_k)_{0 \leq k \leq N},(U_k)_{1 \leq k \leq N})$ connecting
$0$ to $p$ is said to be \textbf{minimal} if it also satisfies
the following:
\begin{itemize}
\item For any propagating terrace
$T'= ((q_k)_{0 \leq k \leq N'},(V_k)_{1 \leq k \leq N'})$
connecting $0$ to $p$, one has that
\[
\{ p_k \mid 0 \leq k \leq N  \}\;  \subset \;
\{ q_k \mid 0 \leq k \leq N' \}.
\]
\item For each $1\leq k\leq N$, the traveling wave $U_k$ is
steeper than any other traveling wave connecting $p_k$ to
$p_{k-1}$.
\end{itemize}
\end{definition}

Roughly speaking, a propagating terrace can be pictured as a
layer of several traveling fronts going at various speeds, the
lower the faster (Figure~\ref{fig:terrace}).
\begin{figure}[h]
\centering
\includegraphics[width=1\textwidth]{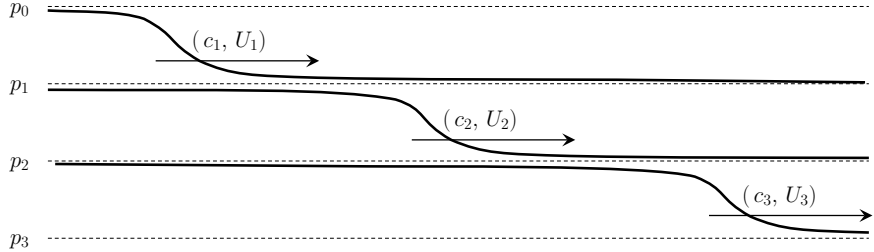}
\caption{A three-step terrace}
\label{fig:terrace}
\end{figure}

The aim of the present paper is to show that the solution of
\eqref{eqn1} with Heaviside-type initial data will converge
to a minimal propagating terrace, as illustrated in
Figure~\ref{fig:terrace-like}.

\begin{figure}[h]
\centering
\includegraphics[width=1\textwidth]{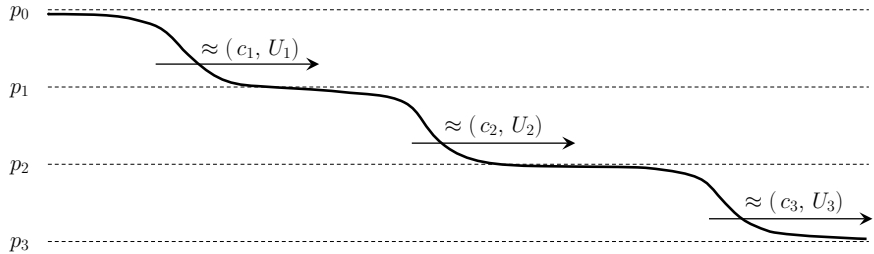}
\caption{Terrace-shaped profile of propagation}
\label{fig:terrace-like}
\end{figure}

In some standard problems such as the KPP and the bistable
equations, the terrace actually consists of a single front
(that is, $N=1$), which means that the solution will eventually
look like a single traveling wave; see Theorem \ref{heavip1}.
However, in more general equations, one cannot expect such simple
dynamics, and this is where the notion of terrace plays a
fundamental role.

The existence of a multi-step terrace has been known in the
spatially homogeneous case (where $f=f(u)$).  Let us give a simple
example.  Consider $f$ as in Figure~\ref{fig:f_terrace} (left),
which is KPP on $[0,\theta_1]$, and bistable on $[\theta_1, 1]$.
The speed of the upper part of the solution is bounded from
above by the speed, say $c$, of the traveling wave for the
bistable nonlinearity $f_{| [\theta_1,1]}$. On the other hand,
the lower part of the solution is pushed from behind by a
spreading front for the KPP nonlinearity $f_{| [0,\theta_1]}$,
whose speed is known to approach $c_*:=2 \sqrt{f'(0)}$.
Therefore, if $c_*>c$,
the upper and lower parts
of the solution necessarily move at two distinct speeds.

\begin{figure}[h!]\label{fig:fterrace}
\vspace{-5mm}
\centering
\includegraphics[width=0.9\textwidth]{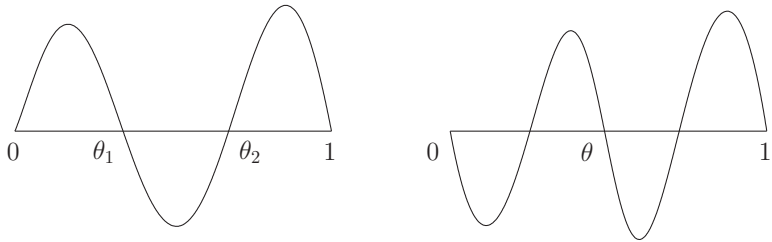}
\caption{Examples of~$f(u)$ that produce a multi-step terrace}
\label{fig:f_terrace}
\end{figure}

Another example was exhibited by Fife and McLeod in~\cite{FifMcL77},
where they considered a specific case of a non-degenerate tristable
nonlinearity, that is, when $f_{|[0,\theta]}$ and $f_{|[\theta,1]}$
are both bistable for some $\theta \in (0,1)$, and that $f'(0)$,
$f'(\theta)$ and $f'(1)$ are all strictly negative; see
Figure~\ref{fig:f_terrace} (right). They showed
that if the speed of the upper bistable part is smaller than the
speed of the lower bistable part, then there does not exist any
single front connecting $0$ to $1$. Furthermore, some solutions
of the Cauchy problem, in particular for Heaviside type initial
data, converge to a combination of those two fronts.
This may be seen as an early study of a propagating terrace for
some very specific examples.

Although the method in \cite{FifMcL77} was expected to hold for
homogeneous nonlinearities composed of a finite number of
bistable parts, it relied strongly on the particular shape of $f$,
and on the non-degeneracy of the equilibria. This means that they
needed some important a priori knowledge on the shape of the
nonlinearity, which we do not need in the present paper.
More importantly, what makes our work different from
those early observations is that we are not simply giving examples
of propagating terraces but are establishing the ubiquity of such
terraces for large classes of reaction nonlinearities, thus
showing that the notion of propagating terrace is fundamental
for studying the dynamics of fronts in general reaction-diffusion
equations.

\subsection{Main results}

We consider solutions of
\eqref{eqn1}-\eqref{initial} whose initial data are given in the form
\begin{equation}\label{eqn:iniheavi}
u_0 (x) = p(x) H(a-x),
\end{equation}
where $a \in \R$ is any constant, and $H$ denotes the Heaviside
function, which is defined by
\[
H(x)=
\left\{
\begin{array}{cr}
0 & \mbox{ if } x<0,\\
1 & \mbox{ if } x\geq 0.\\
\end{array}
\right.
\]
Hereafter, for each $a\in\R$, we denote by $\widehat{u}(t,x;a)$
such solutions. We will prove that $\widehat{u}$ converges in
some sense to a minimal propagating terrace as $t\to\infty$.

\begin{theorem}[Existence of a minimal terrace]
\label{heavip2}
Let Assumption~\ref{assumption-p1'} hold.  Then there exists
a propagating terrace
$((p_k)_{0 \leq k \leq N},(U_k)_{1\leq k\leq N})$
that is minimal in the sense of Definition~\ref{def:terrace}.
Such a minimal propagating terrace is unique, in the sense that
any minimal propagating terrace shares the same~$(p_k)_k$ and that
$U_k$ is unique up to time-shift for each $k$. Moreover,
it satisfies:
\begin{enumerate}[(i)]
\item For any $0 \leq k < N$, the $L$-periodic stationary solution
$p_k$ is isolated and stable from below with respect to
\eqref{eqn1-per}.
\item All the $p_k$ and $U_k$ are steeper than any other entire
solution of \eqref{eqn1}.
\end{enumerate}
\end{theorem}

The existence of a minimal terrace as stated in the above theorem
gives various useful information about the qualitative properties of
the equation. For example, statement~$(ii)$ implies, in particular,
that there exists no traveling wave that intersects any of the~$p_k$.
Note that, in the spatially homogeneous case (namely, $f=f(u)$),
the stability of $p_k$ in statement $(i)$ implies that each $p_k$
is a constant; hence the terrace consists of flat steps.

We now state our convergence result:

\begin{theorem}[Convergence to a minimal terrace]
\label{heavip2-CV}
\sloppy
Let Assumption~\ref{assumption-p1'} hold. Then for any
$a \in \mathbb{R}$, the solution $\widehat{u} (t,x;a)$
converges as $t \rightarrow +\infty$ to the minimal propagating
terrace $((p_k)_{0 \leq k \leq N},(U_k)_{1\leq k\leq N})$ in the
following sense:
\begin{enumerate}[(i)]
\item There exist functions $(m_k (t))_{1 \leq k \leq N}$ with
$m_k (t) = o(t)$ as $t \rightarrow +\infty$ such that
\begin{equation}\label{eqn:CVcor1}
\begin{split}
\widehat{u}(t,x +c_k (t-m_k (t)) ;a)- U_k (t-m_k (t),x +
c_k (t-m_k (t))) \quad \\
\to 0 \ \ \ \hbox{as} \ \ t\to +\infty,
\end{split}
\end{equation}
locally uniformly on $\R$, $c_k$ being the speed of $U_k$.
\item For any $\delta >0$, there exists $C>0$ such that, for any
$1 \leq k \leq N-1$,
$$
\| \widehat{u}(t,\cdot   ;a)- p_k (\cdot) \|_{L^\infty ([c_k
(t-m_k (t))  + C , c_{k+1} (t-m_{k+1} (t)) - C])} \leq \delta
\ \ \hbox{as}\ t\to +\infty,
$$
together with
$$
\| \widehat{u}(t,\cdot +c_1 (t-m_1 (t)) ;a)- p(\cdot) \|_{L^\infty
((-\infty,-C])} \leq \delta \ \ \hbox{as}\ t\to +\infty,
$$
$$
\| \widehat{u}(t,\cdot +c_N (t-m_N (t))  ;a) \|_{L^\infty ([C,
+\infty))} \leq \delta \ \ \hbox{as}\ t\to +\infty.\vspace{3pt}
$$
\end{enumerate}
\end{theorem}

Roughly speaking, statements $(i)$ and $(ii)$ of the above theorem
describe, respectively, the ascending part and the stationary part
of the terrace, the latter being flat if $f=f(u)$, as mentioned above.
It should be noted that, under Assumption \ref{assumption-p1'},
there may exist an infinite number of isolated stationary
solutions between $0$ and $p$, but our Theorem \ref{heavip2}
states that only a finite number of layers appear in the
limiting terrace. The solution seems to ignore excessive
complexity of such nonlinearities.

The proofs of Theorems \ref{heavip2} and \ref{heavip2-CV} are,
in a sense, one and the same. In fact, by first showing that the
steepness of the Heaviside type initial data \eqref{eqn:iniheavi}
implies that the limiting profile of the solution is steeper than
any other entire solution, we will then use this fact to prove
the convergence of the
solution to a minimal terrace without assuming the existence
of a terrace. Hence it automatically implies the existence of a
minimal terrace.

In the special case where Assumption~\ref{assumption-p2'} also
holds, the above two theorems reduce to the following result on
pulsating traveling waves.
Thus it gives a new and highly original
proof for the existence of pulsating traveling waves.

\begin{theorem}[Monostable/bistable/combustion cases]
\label{heavip1}
Let Assumptions \ref{assumption-p1'} and~\ref{assumption-p2'} hold.
Then there exists a pulsating traveling wave $U^* (t,x)$ connecting
$0$ to $p$ with speed $c >0$ that is steeper than any other
entire solution between $0$ and $p$.
Furthermore,
for any $a \in \mathbb{R}$, there exists a function $m(t)$ with
$m(t) = o(t)$ as $t\rightarrow +\infty$ such that
\begin{equation}\label{eqn:CVcor0}
\| \widehat{u}(t,\cdot  ;a)- U^* (t -m(t),\cdot)\|_{L^\infty (\R)}
\to 0 \ \ \hbox{as}\ \ t\to +\infty.
\end{equation}
\end{theorem}

Let us make some comments on Theorem \ref{heavip1}, which is a
special case of the previous two theorems.  As regards
the existence part, there have been earlier studies of the
existence of pulsating traveling waves for monostable (possibly
degenerate) and combustion cases \cite{Ber02-per}, as well as
for some periodic bistable case \cite{FZ11}.  In contrast to those
earlier results, which are derived by various different methods
depending on the type of nonlinearities, our theorem relies on
a new, unified and rather straightforward proof, thus avoiding
to deal directly with the particular features and difficulties
of each case.

As mentioned earlier, the convergence result \eqref{eqn:CVcor0}
for fast decaying initial data such as our Heaviside type ones was
previously known only in the homogeneous setting.
Indeed, even in the standard KPP case with spatially periodic
coefficients, almost all the convergence results in
the literature are concerned with solutions whose initial data
have roughly the same decay rate as one of the traveling waves near
$x=+\infty$.  Only very recently, there are some works in progress
that are trying to deal with
some fast decaying initial data~\cite{MR??,Nadin??}.

Note that, though our results cover a large class of equations,
they are concerned only with a specific type of inital data
\eqref{eqn:iniheavi}.  However, by analogy with the homogeneous
case (see for instance the proofs in~\cite{Uchi78}), we expect
that similar convergence results hold for more general initial
data, such as compactly supported ones.  This is a topic of
particular relevance from an applied point of view, and we will
give a partial answer to this question in a forthcoming
paper~\cite{DGM2}.

\paragraph{Plan of the paper}
Our paper is organized as follows. In Section~\ref{sec:preliminaries},
we will present some preliminaries. As our proofs largely rely on
the so-called intersection number (or the zero number) argument,
we will first give its precise definition and basic properties.
We will then use this method to prove our fundamental lemma
(Lemma \ref{lem:omega-steep}), which roughly states the following:

\begin{quote}
{\it
{\bf Fundamental lemma:} Any function that appears in the
$\omega$-limit set of~$\widehat{u}$ is steeper than any other
entire solution, where $\widehat u$ is the solution of 
\eqref{eqn1} for the initial data \eqref{eqn:iniheavi}
}
\end{quote}

One immediate consequence of the above fundamental lemma is that 
any two elements of the $\omega$-limit set are steeper than each 
other. This means that they are either identical (up to
time shift) or strictly ordered, that is, one is above or below
the other. This observation is important both for establishing
convergence results and for the construction of a multi-step 
terrace.

Here we note that the definition of the $\omega$-limit set in
this paper is slightly different from the standard one, in that
we consider arbitrary spatial translations while taking the
limit as $t_k\to\infty$; see Definition \ref{def:omegalimit}.
The reason for adopting this slightly non-standard defintion is
that, since each step of the terrace moves at a different speed,
we cannot capture the asymptotic profile of the solution in
a single frame.  A multi-speed observation is unavoidable in the
case of a multi-step terrace.
In the last part of Section~\ref{sec:preliminaries}, we will
prove a lemma on the spreading speed of~$\widehat{u}$, which
will be used repeatedly in later sections but is also of
independent interest in its own right.

In Section~\ref{sec:lemmas}, we will use our fundamental lemma
to prove the convergence of solutions with
Heaviside type initial data to a unique limit around any level
set. The entire solution
thereby constructed possesses, in some sense, some qualitative
properties of traveling fronts, such as monotonicity in time.
The same result could in fact be shown with a similar argument
in a more general setting without the periodicity assumption.

We will then show in Section~\ref{sec:periodic} that this limit
is a pulsating traveling wave
connecting some pair of $L$-periodic stationary solutions
$p_-< p_+$ that lie between $0$ and $p$.  Once again the
above-mentioned fundamental lemma plays a key role in deriving
this result. This leads to construction of a multi-step minimal
terrace inductively, as described in Theorems~\ref{heavip2}
and~\ref{heavip2-CV}.  In the special case where
Assumtion~\ref{assumption-p2'} holds, we have $p_-=0,\,p_+=p$,
thus the terrace is a single traveling wave. 
As the existence of any traveling wave is not a priori assumed,
this leads to both the existence and the convergence results in
Theorem~\ref{heavip1}.


\section{Preliminaries}\label{sec:preliminaries}

\subsection{Zero number}

Our proof of the main results relies strongly on a zero-number
argument. The application of this argument --- or the
``Sturmian principle" --- to the convergence proof in semilinear
parabolic equations first appeared in \cite{Matano78}. But what makes
the present paper different from earlier work is that we employ
the zero-number argument to prove not only the convergence but also
the existence of the target objects, namely the terrace and
pulsating traveling waves.

In this paper, besides the standard zero-number $Z[\,\cdot\, ]$,
we introduce a related notion $SGN[\,\cdot\,]$, which turns out to
be exceedingly useful in establishing our fundamental
lemma.

\begin{definition}\label{def:zeros}
For any real-valued function $w$ on $\R$, we define:
\begin{itemize}
\item $Z \left[ w(\cdot) \right]$ is the number of sign changes of
$w$, namely the supremum over all $k \in \mathbb{N}$ such that
there exist real numbers $x_1 < x_2 <... < x_{k+1}$ with
\[
w(x_i) . w(x_{i+1})< 0 \ \ \hbox{for all}\ \ i=1,2,...,k .
\]
We set $Z[w]=-1$ if $w\equiv 0$.
\item
$SGN\left[w(\cdot)\right]$, which is defined when
$Z\left[w(\cdot)\right]<\infty$, is the word consisting of $+$ and
$-$ that describes the signs of $w(x_1),\ldots,w(x_{k+1})$, where
$x_1<\cdots<x_{k+1}$ is the sequence that appears in the definition
of $Z[w]$ with maximal $k$. We set $SGN\left[0\right] =[\ ]$, the
empty word.
\end{itemize}
\end{definition}

If $w$ is a smooth function having only simple zeros on $\R$,
then $Z\left[w\right]$ coincides with the number of zeros of $w$.
For example,
\[
Z[x^2-1]=2,\ \ SGN[x^2-1]=[+ - +].
\]

By definition, the length of the word $SGN[w]$ is equal to $Z[w]+1$.
It is also clear that $Z[w]=0$ if and only if either
$w\geq 0,\,w\not\equiv 0$ or $w\leq 0,\,w\not\equiv 0$. The former
implies $SGN[w]=[+]$, and the latter $SGN[w]=[-]$.

If $A, B$ are two words consisting of $+$ and $-$, we write
$A\rhd B$ (or, equivalently, $B\lhd A$) if $B$ is a subword of $A$.
For example,
\[
[+\, -] \rhd B\ \ \hbox{for}\ B=[+\,-],\,[+],\,[-],\,[\ ]
\quad\ \hbox{but not}\ \ [+\,-]\rhd [-\,+].
\]

Let us recall some properties of $Z$ and $SGN$:

\begin{lemma}\label{lem:Z1}
Let $w(t,x) \not \equiv 0$ be a bounded solution of a parabolic
equation of the form
\begin{equation}\label{eqw}
\partial_t w = \partial_{xx} w + c (t,x) w \ \ \mbox{ on a domain }
( t_1 , t_2) \times \R ,
\end{equation}
where $c$ is bounded. Then, for each $t \in (t_1,t_2)$, the zeros
of $w(t,\cdot)$ do not accumulate in $\R$. Furthermore,
\begin{enumerate}[$(i)$]
\item $Z\left[ w(t, \cdot) \right]$ and $SGN\left[w(t,\cdot)\right]$
are nonincreasing in $t$, that is, for any $t'>t$,
\[
 Z\left[w(t,\cdot)\right]\geq Z\left[w(t',\cdot)\right],\quad\
SGN\left[w(t,\cdot)\right] \rhd SGN\left[w(t',\cdot)\right];
\]
here the assertion remains true even for $t=t_1$ if $w$ can
be extended to a continuous function on $[t_1 , t_2) \times \R$;
\item if $w(t',x')=\partial_x w(t',x')=0$ for some
$t' \in (t_1 ,t_2)$
and $x'\in\R$, then
\[
Z \left[ w(t,\cdot) \right]-2 \geq Z\left[ w(s,\cdot) \right]\geq 0 \
\ \ \hbox{for any}\ t \in (t_1, t') \ \hbox{and}\ s
\in (t' ,t_2)
\]
whenever $Z \left[ w(t,\cdot) \right] < \infty$.
\end{enumerate}
\end{lemma}

The second inequality of statement $(ii)$ above implies that,
for any $t\in(t_1,t_2)$, the function $w(x,t)$ does not vanish
entirely on $\R$ unless $w\equiv 0$ on $(t_1,t_2)\times\R$.
Statement $(ii)$ is due to \cite{An88}, where this result
is proved by using similarity variables and
expansion by Hermitian polynomials.
Though \cite{An88} deals with only equations on bounded intervals,
the result can easily be extended to $\R$ by applying the maximun
principle near $x=\pm\infty$; see \cite{DuMatano10}.

The statement $(i)$ for $Z[w]$ follows from $(ii)$, at least when
the domain is a bounded interval, but it can be shown more directly
by a combination of the maximum principle and a topological argument
similar to the Jordan curve theorem (which is a more standard way
to prove this statement). In fact, this direct proof proves the assertion for $SGN[w]$, from which the assertion for $Z[w]$
follows automatically; see, for example, \cite{Matano78} for a
similar argument.

One can also check that $Z$ is semi-continuous with respect to
the pointwise convergence, that is:

\begin{lemma}\label{lem:Z-lim}
Let $(w_n)_{n \in \mathbb{N}}$ be a sequence of functions converging
to $w$ pointwise on $\R$. Then
\[
\ w \equiv 0 \ \mbox{ or } \ Z\left[ w \right] \leq \liminf_{n
\rightarrow \infty} Z\left[ w_n \right],
\]
\[
SGN\left[w\right] \lhd \liminf_{n \rightarrow \infty} SGN [w_n].
\]
\end{lemma}

Combining the above two lemmas, we obtain the following lemma:

\begin{lemma}\label{lem:Z(u1-u2)}
Let $u_1$ and $u_2$ be solutions of \eqref{eqn1} such that the
initial data $u_1(0,x)$ is a piecewise continuous bouded function
on $\R$, while $u_2(0,x)$ is bounded and continuous on $\R$. Assume
also that $u_1(0,x)- u_2(0,x)$ changes sign at most finitely many
times on $\R$.  Then
\begin{enumerate}[$(i)$]
\item for any $0\leq t <t'<\infty$,
\begin{equation}\label{Z:t=0}
\begin{split}
Z \left[ u_1 (t,\cdot) - u_2 (t,\cdot)\right] & \geq
Z \left[ u_1 (t',\cdot) - u_2 (t',\cdot)\right],\\
SGN \left[ u_1 (t,\cdot) - u_2 (t,\cdot)\right] & \rhd
SGN \left[ u_1 (t',\cdot) - u_2 (t',\cdot)\right];
\end{split}
\end{equation}
\item if, for some $t'>0$, the graph of $u_1(x,t')$ and that of
$u_2(x,t')$ are tangential at some point in $\R$, and if
$u_1\not\equiv u_2$, then for any
$t,s$ with $0\leq t < t'<s$,
\[
Z \left[ u_1 (t,\cdot) - u_2 (t,\cdot)\right] -2 \geq
Z \left[ u_1 (s,\cdot) - u_2 (s,\cdot)\right]\geq 0.
\]
The same conclusion holds if $u_1, u_2$ are entire solutions of 
\eqref{eqn1}, in which case $t'\in\R$ is arbitrary and 
$-\infty<t<t'<s<\infty$.
\end{enumerate}
\end{lemma}

\begin{proof}
The function $w:=u_1-u_2$ satisfies an equation of the form
\eqref{eqw} on $(0,\infty)\times\R$ with
$c(x,t):=(f(x,u_1)-f(x,u_2))/(u_1-u_2)$ being bounded. Thus the
conclusion of the lemma follows from Lemma~\ref{lem:Z1} except
for \eqref{Z:t=0} with $t=0$.  Moreover statement \eqref{Z:t=0}
with $t=0$ also follows from Lemma~\ref{lem:Z1} if $u_1(0,x)$
and $u_2(0,x)$ are both continuous.  In the general
case where $u_1(0,x)$ is only piecewise continuous, we approximate
$u_1$ by a sequence of solutions of \eqref{eqn1}, say $u_{1,n}$,
whose initial data $u_{1,n}(0,x)$ are continuous and satisfy
\begin{itemize}
\item[(a)] $\sup_n\Vert u_{1,n}(0,\cdot)\Vert_{L^\infty(\R)}<\infty$
and $u_{1,n}(0,x)\to u_1(0,x)$ pointwise on $\R$;
\item[(b)] $SGN[u_{1,n}(0,\cdot)-u_2(0,\cdot)]=SGN[u_1(0,\cdot)
-u_2(0,\cdot)]$ \ for $n=1,2,3,\ldots$.
\end{itemize}
Then we have, for each $t'>0$,
\[
SGN \left[ u_1 (0,\cdot) - u_2 (0,\cdot)\right] \rhd
SGN \left[ u_{1,n} (t',\cdot) - u_2 (t',\cdot)\right].
\]
Letting $n\to\infty$ and applying Lemma \ref{lem:Z-lim}, we obtain
the desired conclusion.
\end{proof}
\\

The following corollary will be used repeatedly later:

\begin{corollary}\label{cor:steeper}
Let $v_1,\,v_2$ be two entire solutions of \eqref{eqn1}, and assume
that
\[
SGN\left[v_1(t_1,\cdot)-v_2(t_2,\cdot)\right]\lhd [+\;-]
\quad \hbox{for any}\ \ t_1, t_2 \in\R.
\]
Then $v_1$ is steeper than $v_2$ in the sense of
Definition~\ref{def:steep}.
\end{corollary}

\begin{proof}
Fix $t_1,t_2\in\R$ arbitrarily. From the assumption we see that
\[
Z \left[ v_1 (t+t_1,\cdot) - v_2 (t+t_2,\cdot)\right] \leq 1 \quad
\ \hbox{for all}\ \ t\in \R.
\]
If $v_1(\cdot+t_1,\cdot)\not\equiv v_2(\cdot+t_2,\cdot)$, then by 
Lemma~\ref{lem:Z(u1-u2)} (ii), the function $v_1(t_1,x)-v_2(t_2,x)$ 
has at most one zero on $\R$, and that this zero is simple. Let 
$x_1$ be such a zero; that is, $v_1(t_1,x_1)=v_2(t_2,x_1)$. Then 
the simplicity of this zero and the sign property 
$SGN\left[v_1-v_2\right]\lhd [+\;-]$ imply that 
$\partial_x v_1(t_1,x_1)<\partial_x v_2(t_2,x_1)$. This proves 
that $v_1$ is steeper than $v_2$.
\end{proof}

\subsection{Fundamental lemma on the $\omega$-limit set of
$\widehat{u}$}

The following definition of the $\omega$-limit set of a solution
$u$ is slightly different from the standard one, as we add
arbitrary spatial translations while taking the long-time limit.
The reason for adopting this definition is that, since each step
of the terrace moves at a different speed, we need multi-speed
observations in order to fully capture the asymptotic profile of
the solution.

\begin{definition}\label{def:omegalimit}
Let $u(t,x)$ be any bounded solution of Cauchy problem
\eqref{eqn1}-\eqref{initial}. We call $v(t,x)$ an
\textbf{$\boldsymbol{\omega}$-limit orbit} of $u$ if there exist
two sequences $t_j \rightarrow +\infty$ and $k_j \in \mathbb{Z}$
such that
\[
u(t+t_j,x+k_j L) \rightarrow v(t,x)\ \mbox{ as } j \to +\infty
\; \mbox{ locally uniformly on } \R.
\]
\end{definition}

\begin{remark}
By parabolic estimates, the above convergence takes place in
$C^2$ in $x$ and $C^1$ in $t$.  Hence
one can easily check that  any $\omega$-limit orbit of $u$ is
an entire solution of \eqref{eqn1}. Moreover, if $v(t,x)$ is an
$\omega$-limit orbit of $u$, then so is $v(t+\tau,x+kL)$ for any
$\tau \in \R$ and $k\in \mathbb{Z}$.
\end{remark}

Let us now state a fundamental lemma that will be used repeatedly
throughout our paper:

\begin{lemma}\label{lem:omega-steep}
Let $a \in \R$ and let $v_1$ be any $\omega$-limit orbit of
$\widehat{u} (t,x;a)$. Then $v_1$ is steeper than any entire
solution of \eqref{eqn1} in the sense of Definition \ref{def:steep},
provided that this entire solution lies between $0$ and $p$.
\end{lemma}

\begin{proof}
Fix $a \in \R$, and let the sequences $t_j \rightarrow + \infty$
and $k_j \in \mathbb{Z}$ be such that $u(t+t_j,x+k_jL)\to v_1(t,x)$
locally uniformly as $j \to + \infty$. By standard parabolic
estimates, the convergence in fact holds in
$C^1_{loc} (\mathbb{R}^2)$.

Let $v$ be any entire solution lying between 0 and $p$. Since $0\leq v(t,x)\leq p(x)$, we have
$\widehat u(0,x;a)\geq v(t,x)$ for $x<a$ and
$\widehat u(0,x;a)\leq v(t,x)$ for $x>a$.  Consequently, for any
$j\in {\mathbb N}$ and $\tau\in\R$,
\[
Z \left[ \widehat{u} (0,\cdot;a) - v(\tau-t_j,\cdot) \right]=1 \ \
\hbox{and}\ \ SGN\left[ \widehat{u} (0,\cdot;a) - v(\tau-t_j,\cdot)
\right]=[+ \; -].
\]
It follows from Lemma~\ref{lem:Z(u1-u2)} that, for all 
$j \in \mathbb{N}$ and $t \geq -t_j$,
\[
Z \left[ \widehat{u} (t + t_j,\cdot;a) - v (t+\tau, \cdot)\right]
\leq 1,
\]
\[
SGN\left[ \widehat{u} (t + t_j,\cdot;a) - v (t+\tau, \cdot)\right]
\lhd [+ \; -].
\]
Passing to the limit as $j \rightarrow +\infty$, we get
\[
SGN \left[ v_1(t ,\cdot) - v(t +\tau, \cdot)\right] \lhd [+ \; -]
\quad\ \hbox{for any}\ \ t,\tau\in\R.
\]
Hence, by Corollary \ref{cor:steeper}, $v_1$ is steeper than $v$
in the sense of Definition~\ref{def:steep}.
\end{proof}

\subsection{Spreading of the solution with positive speed}

Before going to the proof of our main results, we investigate
some spreading property of solutions of \eqref{eqn1}. The result
below, which can be of independent interest, will be used repeatedly
later. Recall that $\widehat{u} (t,x;a)$ is the solution of
\eqref{eqn1} with initial data $u_0 (x) = p(x) H(a-x)$.

\begin{lemma}\label{spread}
Let Assumption~\ref{assumption-p1'} be satisfied. Then there exist
constants $0<c_* < c^* < +\infty$ that do not depend on $a$, such that
\begin{itemize}
\item [(i)] for each $c>c^*$,  one has:
$\displaystyle \lim_{t\to\infty} \sup_{x\geq ct}\widehat{u}
(t,x;a)=0$;
\item [(ii)] for each $c\in \left(0,c_* \right)$ one has
\begin{equation*}
\lim_{t\to\infty} \sup_{x\leq ct}|\widehat{u} (t,x;a)- p(x)|=0.
\end{equation*}
\end{itemize}
\end{lemma}

\begin{proof}
Note that from the $C^1$-regularity and the periodicity of $f$,
there exists $K>0$ such that for any $x\in \R$ and
$0 \leq u \leq \sup_{x \in \R} p(x)$, we have that
$f(x,u) \leq Ku$. Now define
\[
\overline{u}(t,x) := e^{-\sqrt{K} (x-a -2\sqrt{K} t)} \| p\|_\infty,
\]
which is a solution of the linear homogeneous equation
\[
\overline{u}_t- \partial_{xx} \overline{u}=K\overline{u}.
\]
Then $\overline u$ is a supersolution of \eqref{eqn1} in the range
$0 \leq u \leq \sup_{x \in \R} p(x)$. Since
$\overline{u}(0,x)\geq\|p\|_\infty \geq p(x) = \widehat{u}(0,x;a)$
for all $x \leq a$, while
$\overline{u}(0,x) \geq 0 = \widehat{u} (0,x;a)$ for all $x > a$,
it follows from the comparison principle that for all
$t \geq 0$ and $x \in \R$,
\[
\widehat{u} (t,x;a) \leq \overline{u}(t,x).
\]
Therefore, for any speed $c > 2 \sqrt{K}$,
\[
\widehat{u}(t,x+ct;a) \leq \overline{u}(t,x+ct) \rightarrow 0
\]
uniformly with respect to $x \geq 0$ as $t\rightarrow +\infty$.

Let us now find a positive lower bound for the spreading speed.
Let $u_0$ be the compactly supported function given in
Assumption~\ref{assumption-p1'}. This means that the solution
$\underline{u}$ of the Cauchy problem \eqref{eqn1}-\eqref{initial}
with initial data $0 \leq u_0 < p$ converges locally uniformly to
$p$ as $t \to +\infty$. Thanks to the periodicity of \eqref{eqn1},
one can assume without loss of generality that
\[
supp (u_0) \subset \left[a-C,a\right]
\]
for some $C>0$.  Since $\underline u\to p$ as $t\to\infty$, there
exists $T >0$ such that
\begin{equation*}
\underline{u}(T,x) \geq \max \left\{u_0 (x), u_0 (x-L) \right\}
\ \ \ \text{for any $x \in \R$
}.
\end{equation*}
By the comparison principle, it follows that
\begin{eqnarray*}
\underline{u}(2T,x) & \geq & \max \left\{ \underline{u}(T,x),\underline{u}(T,x-L)\right\} \\
& \geq & \max \left\{ u_0 (x), u_0 (x-L), u_0 (x-2L) \right\}.
\end{eqnarray*}
By induction, we obtain that for all $k \in \mathbb{N}$,
$$\underline{u} (kT,x) \geq \max \left\{u_0 (x-jL) \ | \ j \in \mathbb{N}, \ 0 \leq j \leq k \right\}.$$
Since $\widehat{u} (0,x;a) = p(x)$ for all $x \in (-\infty, a]$, we have that $$\widehat{u} (0,x;a) \geq \max \left\{u_0 (x+jL) \ | \ j \in \mathbb{N} \right\}.$$
Applying the comparison principle, one gets for all $t>0$ and $x\in \R$ that
$$\widehat{u} (t,x;a) \geq \max \left\{\underline{u} (t,x+jL) \ | \ j \in \mathbb{N} \right\}.$$
Hence, for all $x \in \R$,
$$\widehat{u} (kT,x;a) \geq \max \left\{u_0 (x-jL) \ | \ j \in \mathbb{Z}, \ j \leq k \right\}.$$
Therefore, from Assumption~\ref{assumption-p1'}, we have that for any $k \in \mathbb{N}$,
\begin{eqnarray*}
\widehat{u} (\tau + kT,x;a) & \geq & \max \left\{\underline{u} (\tau,x-jL) \ | \ j \in \mathbb{Z}, \ j \leq k \right\}, \vspace{3pt}\\
& \longrightarrow & p(x),
\end{eqnarray*}
where the convergence holds as $\tau \rightarrow +\infty$, uniformly with respect to $k~\in~\mathbb{N}$ and $x~\in~(-\infty,kL]$.

Let us now define $c_*= L / T>0$ and choose any $c$ with $0< c <c_*$.
Denote by~$\lceil y \rceil$ the ceiling function of $y$, that is,
the least integer not smaller than $y$. Then for any~$t \geq 0$, let
$$\tau (t) := t - \left\lceil \frac{ct}{L} \right\rceil T.$$
As $c < c_* = L/T$, one can easily check that $\tau\to +\infty$
as $t \rightarrow +\infty$. Thus,
$$\sup_{x \leq \left\lceil \frac{ct}{L} \right\rceil L} \left| \widehat{u} \left( \tau (t) +  \left\lceil \frac{ct}{L} \right\rceil T,x;a\right) - p(x) \right| \longrightarrow 0 \ \mbox{ as } t \rightarrow +\infty,$$
and, since $ct \leq \left\lceil \frac{ct}{L} \right\rceil L$ and $t =\tau(t) +\left\lceil \frac{ct}{L} \right\rceil T$ for all $t \geq 0$,
$$\sup_{x \leq ct} \left| \widehat{u} (t,x;a) - p(x) \right| \longrightarrow 0 \ \mbox{ as } t \rightarrow +\infty, $$
which concludes the proof of Lemma~\ref{spread}.~\end{proof}

\section{Convergence of the solutions with shifted initial data}\label{sec:lemmas}

In this section, we aim to prove an important lemma on the convergence of the solutions of the Cauchy problem with some shifted Heaviside type initial data around a given level set.

\begin{lemma}\label{heavi1lemma}
Let Assumption \ref{assumption-p1'} be satisfied. Let $x_0 \in \R$ be given. For any $0 < \alpha < p(x_0)$ and $a <x_0$, let us define
\begin{equation}\label{DEF-tau}
\tau (x_0,\alpha,a) := \min \left\{ t>0 \; | \ \widehat{u} (t,x_0;a) = \alpha \right\}.
\end{equation}
Then the following limit exists for the topology of $C_{loc}^1 (\R^2)$
\begin{equation}\label{eqn:limexist}
\lim_{a \rightarrow -\infty} \widehat{u} (t+\tau (x_0,\alpha,a),x;a):=w_\infty (t,x;\alpha).
\end{equation}
Function $w_\infty(t,x;\alpha)$ is an entire solution of \eqref{eqn1} that is steeper than any other entire solution.
Furthermore, the following alternative holds true: either it is a stationary solution, or $\partial_t w_\infty (t,x;\alpha) > 0$ for all $ (t,x) \in \R^2$.  The former assertion is impossible for each $\alpha$ close enough to $p(x_0)$.
\end{lemma}

This lemma states that if we look at some well chosen level set for shifted initial data, the profile of the solution locally converges to a monotonically increasing entire solution~$w_\infty$ of \eqref{eqn1}, which connects two stationary solutions. In the sequel we will show that $w_\infty$ is a pulsating traveling wave.

A similar result holds in a nonperiodic heterogeneous framework, where one could show that $\widehat{u} (t+\tau (x_0,\alpha,a),x;a)$ converges as $a \rightarrow -\infty$ to the steepest entire solution of \eqref{eqn1} taking the value $\alpha$ at the point $(0,x_0)$. However, as the notion of traveling wave is not clear in general, we chose to restrict ourselves to the more standard periodic setting.\\

The proof of this result is split into three parts.
Before we begin the proof of this lemma, let us make the following remark, explaining the choice of such shifted initial data.
\begin{remark}\label{rem:period}
Notice that one has
\begin{equation*}
\widehat{u}(t,x;a+L)=\widehat{u}(t,x-L;a),\;\;\forall (t,x,a)\in [0,\infty)\times\R\times\R.
\end{equation*}
This implies that for any $0< \alpha  < p(x_0)$ and $k \in \mathbb{N}$:
\begin{equation*}
\tau (x_0, \alpha,a-kL)=\tau (x_0 +L, \alpha,a-(k-1)L).
\end{equation*}
Now since the initial data $p(x)H(a-x)$ is increasing with respect to $a$, the comparison principle provides that for each given $(t,x)\in\R^2$,  the maps $a\mapsto \widehat{u}(t,x;a)$ and $k \mapsto \tau (x_0, \alpha,a-kL)$ are nondecreasing.

Let $x_0 \in \R$ and $a < x_0$ be given. Then function $w_\infty$ defined in Lemma~\ref{heavi1lemma} rewrites as
\begin{eqnarray*}
w_\infty (t,x;\alpha) & = & \lim_{k \rightarrow +\infty} \widehat{u} \left(t+ \tau(x_0, \alpha, a-kL),x;a-kL \right) \vspace{3pt}\\
& =& \lim_{k \rightarrow +\infty} \widehat{u} \left(t+ \tau(x_0, \alpha, a-kL),x+kL;a \right).
\end{eqnarray*}
As it is clear that $\tau(x_0, \alpha, a-kL) \rightarrow +\infty$ as $k \rightarrow +\infty$, the above computations explain the choice of the shifts of the initial data in order to study the large time behavior of the solution of the Cauchy problem.
\end{remark}

\subsection{Existence of $\tau (x_0,\alpha,a)$ and $w_\infty$}

Let us first recall that Assumption \ref{assumption-p1'} holds true. Let $x_0 \in \R$ and $0 < \alpha < p(x_0)$ be given. Let us first note that for any $a <x_0$, the following quantity exists and is finite
$$\tau (x_0,\alpha,a) := \min \left\{ t>0 | \ \widehat{u} (t,x_0;a) = \alpha \right\} < +\infty.$$
Indeed, let $k$ be some large enough given integer such that $supp \; (u_0 (\cdot +kL)) \subset (-\infty,a]$, wherein $u_0$ is the compactly supported function arising in Assumption~\ref{assumption-p1'}. Recall that the corresponding solution of the Cauchy problem converges locally uniformly to $p$. Since $u_0 (x+kL) < p(x)$ for all $x \leq a$, one has
$$u_0 (\cdot+kL) \leq \widehat{u} (0,\cdot;a) \leq p(\cdot).$$
Therefore, $\widehat{u} (t,x;a)$ converges locally uniformly with respect to $x$ to $p$ as $t \rightarrow +\infty$. In particular, this leads us to $ \widehat{u} (t,x_0;a) \rightarrow p(x) > \alpha$ as $t \rightarrow +\infty$.
On the other hand, since $a<x_0$, we have $\widehat{u} (0,x_0;a) = 0 < \alpha$ and the existence of $\tau (x_0, \alpha,a)$ immediately follows.\\

We now aim  to prove that the following limit exists for all $(t,x) \in \R^2$:
$$
w_\infty (t,x;\alpha) = \lim_{a \rightarrow -\infty} \widehat{u} (t+\tau (x_0,\alpha,a),x;a).$$
To do so, let us first notice that from parabolic estimates, the family of functions $\left\{ \widehat{u} (t+\tau (x_0,\alpha,a),x;a)\right\}_{a< x_0}$ is uniformly bounded along with their derivatives. Therefore it is relatively compact for the topology of $C^1_{loc} (\R^2)$ with respect to  $(t,x)$.

Let $(a_k)_{k \in \mathbb{N}}$ be a given sequence such that $a_k\to -\infty$ as $k \rightarrow +\infty$, and such that the following limit holds true:
$$\widehat{u} (t+\tau (x_0,\alpha,a_k),x;a_k) \rightarrow w_\infty (t,x),$$
as $k \rightarrow +\infty$, wherein $w_\infty$ is some function and where the convergence holds in $C^1_{loc} (\mathbb{R}^2)$. Up to a subsequence, one may assume that $a_k \rightarrow a_\infty$ in $\R_{/L\mathbb{Z}}$. Then, from Remark~\ref{rem:period}, one can check that $w_\infty$ is an $\omega$-limit orbit of $\widehat{u} (t,x,a_\infty)$, and it therefore follows from Lemma~\ref{lem:omega-steep} that it is steeper than any other entire solution in the sense of Definition~\ref{def:steep}.

Recalling Definition~\ref{def:steep}, there is a unique entire solution $w_\infty$ of \eqref{eqn1} that is steeper than any other entire solution, and such that $w_\infty (0,x_0) =\alpha$. It follows that $w_\infty$ does not depend on the choice of the sequence $\{a_k\}$. Finally the relative compactness of the family of functions $\left\{ \widehat{u} (t+\tau (x_0,\alpha,a),x;a)\right\}_{a< x_0}$ completes the proof of the existence of
$$w_\infty (t,x;\alpha) = \lim_{a \rightarrow -\infty} \widehat{u} (t+\tau(x_0,\alpha,a),x;a),$$
together with the convergence for the topology of $C_{loc}^1 (\mathbb{R}^2)$.

\subsection{Monotonicity in time of $w_\infty$}
In order to complete the proof of Lemma \ref{heavi1lemma}, it remains to prove the alternative part.
To do so, we will show that function $w_\infty$ is nondecreasing with respect to time.
We will more precisely prove that for any $t \in \R$, $\partial_t w_\infty (t,\cdot;\alpha)$ does not change sign.
To prove this statement, we will argue by contradiction by assuming that for some given $t_1 \in \R$, there exist $x_1\in\R$ and $x_2\in\R$ such that
\begin{equation}\label{contrat}
\partial_t w_\infty (t_1,x_1;\alpha)>0 \mbox{ and } \partial_t w_\infty (t_1,x_2;\alpha) <0.
\end{equation}
It is then clear that for any $\tau$ small enough, one has
$$Z \left[ w_\infty (t_1 + \tau,\cdot;\alpha)- w_\infty (t_1,\cdot;\alpha) \right] \geq 1.$$
Besides, recall that $w_\infty$ is an $\omega$-limit orbit of $\widehat{u}$, and so is $w_\infty (\cdot+ \tau, \cdot)$ for any $\tau \in \R$. Therefore, they are steeper than each other and it immediately follows that $w_\infty (t_1 +\tau,\cdot;\alpha) \equiv w_\infty (t_1,\cdot;\alpha)$ for each $\tau$ small enough. Hence,
$$\partial_t w_\infty (t_1, \cdot;\alpha) \equiv 0,$$
a contradiction together with \eqref{contrat}.
This implies that for any $t\in\R$, one has
\begin{equation}\label{SGNmono}
SGN \left[ \partial_t w_\infty (t,\cdot;\alpha) \right] = [\; ] \mbox{ or } [+ ] \mbox{ or } [ -].
\end{equation}
Let us denote by $\Phi :=\partial_t w_\infty$. It is an entire solution of the linear parabolic equation
$$\partial_t \Phi = \partial_{xx} \Phi + \partial_u f (x,w_\infty) \Phi.$$
We infer from \eqref{SGNmono} and the strong maximum principle that either $\partial_t w_\infty < 0$, either $\partial_t w_\infty >0$ or $\partial_t w_\infty \equiv 0$.
Next due to the definition of $\tau (x_0,\alpha,a)$, one has $\partial_t w_\infty (0,x_0) \geq 0$.
This completes the proof of the alternative part of Lemma~\ref{heavi1lemma}.

To conclude the proof of Lemma \ref{heavi1lemma}, let us show that when $\alpha$ is chosen close enough to $p(x_0)$ then $w_\infty$ cannot be a stationary solution of \eqref{eqn1}. To show that let us first notice that due to Assumption~\ref{assumption-p1'}, the stationary solution $p$ is  isolated with respect to the other stationary solutions. Therefore, one can choose $\alpha$ close enough to $p(x_0)$ so that there is no stationary solution $q$ with $q(x_0) = \alpha$. Then due to Assumption ~\ref{assumption-p1'}, $w_\infty$ is not a stationary solution and it converges to $p$ as $t \rightarrow +\infty$. This completes the proof of the lemma.


\section{Convergence to a propagating terrace}\label{sec:periodic}

The aim of this section is to prove the convergence of the solutions to a propagating terrace.
Since only small differences will arise depending on whether Assumption~\ref{assumption-p2'} holds or not, we give in this section a common proof for both Theorem~\ref{heavip1} and Theorem~\ref{heavip2}. We will explicitely write down whenever we use Assumption~\ref{assumption-p2'}.

In this section, we will first show that the functions $w_\infty (t,x;\alpha)$, constructed in the previous sections, are either traveling waves or stationary solutions. Using some well chosen values of $\alpha$, we will then be able to construct, by using iterative arguments, the minimal propagating terrace describing the long time behavior of the solution $\widehat{u}$ of $\eqref{eqn1}$ with an Heaviside-type initial data, as stated
in Theorem~\ref{heavip2}. Lastly, we will prove that it satisfies all the required statements.

\subsection{Convergence to a pulsating traveling wave for some level sets}\label{CVpulswave}

Recalling Definition \eqref{DEF-tau}, let us define the sequence $$\tau_k := \left\{ \begin{array}{ll}
\tau (x_0, \alpha,a-kL)-\tau (x_0 , \alpha, a-(k-1)L) & \mbox{ if } k \geq 1 \vspace{3pt}\\
\tau (x_0, \alpha, a) & \mbox{ if } k=0 \end{array} \right. ,$$
so that for all $k \in \mathbb{N}$,
$$\tau (x_0,\alpha, a-kL) = \sum_{i=0}^k \tau_i .$$
Then the following result holds true:
\begin{lemma}
For any $\alpha \in (0,p(x_0)$, the entire solution $w_\infty$ provided by Lemma \ref{heavi1lemma} is either a positive periodic stationary solution, or a pulsating traveling wave.
\end{lemma}

\begin{proof}
The proof of this result relies on some properties of the sequence $\{\tau_k\}$. It is split into two parts.
Let us first assume that there exists some subsequence $(\tau_{k_j})_{j \in \mathbb{N}}$ converging to some $T>0$.
Then we obtain that
\begin{eqnarray*}
w_\infty (t+T,x;\alpha) & = & \lim_{k \rightarrow + \infty} \widehat{u} \left(t+ \tau_k + \tau (x_0, \alpha, a-(k-1)L),x;a-(k-1)L\right)\\
& = & \lim_{k \rightarrow + \infty} \widehat{u} \left(t+ \tau (x_0, \alpha, a-kL),x-L;a-kL\right)\\
& = & w_\infty (t,x-L;\alpha).
\end{eqnarray*}
Moreover, Lemma~\ref{heavi1lemma} provides that $\partial_t w_\infty \geq 0$ and therefore it converges as $t \rightarrow \pm \infty$ to two periodic stationary solutions $p_{\pm}$ (the periodicity follows from the above equality). If the two functions $p_{+}$ and $p_{-}$ are distinct, then $w_\infty$ is a pulsating traveling wave. If they are identically equal, then $w_\infty$ is a periodic stationary solution. Furthermore, it is positive since $w_\infty (0,x_0;\alpha) = \alpha >0$ and the strong maximum principle.\\

Let us now consider the case when no subsequence of $(\tau_k)_k$ converges to some positive constant and let us show that $w_\infty$ is stationary. Due to Remark~\ref{rem:period}, it is clear that for all $k \in \mathbb{N}$, $\tau_k \geq 0$. On the other hand, it follows from the spreading speed property provided by Lemma~\ref{spread} that
\begin{equation}\label{estim:tausum}
\frac{L}{c^*} \leq \liminf_{k\to\infty}\frac{ \tau(x_0, \alpha, a-kL)}{k} \leq \limsup_{k\to\infty}\frac{\tau(x_0, \alpha, a-kL)}{k}  \leq \frac{L}{c_*}.
\end{equation}
Therefore, since no subsequence of $(\tau_k)_k$ converge to some positive constant, \eqref{estim:tausum} implies that one can find two subsequences converging respectively to $0$ and $+\infty$.

By considering a subsequence converging to $0$, the same computations as above with $T = 0$ lead us to
\begin{eqnarray}\label{winfty-periodic}
w_\infty (t,x;\alpha) & = & w_\infty (t,x-L;\alpha),
\end{eqnarray}
and function $w_\infty$ is $L$-periodic with respect to the space variable for all time.

In order to show that $w_\infty$ is stationary, let us argue by contradiction by assuming that $w_\infty$ is not stationary.
Then using Lemma~\ref{heavi1lemma}, one has that
$$\partial_t w_\infty (0,x_0+L;\alpha) >0.$$
Thanks to the $C_{loc}^1$ convergence of $\widehat{u}(\cdot+\tau(x_0, \alpha,a),\cdot;a)$ to $w_\infty$ as $a \rightarrow -\infty$, there exists some $\delta >0$ such that for any $0 \leq t \leq \delta $ and $a$ large enough, one has
$$\partial_t \widehat{u} (t+\tau(x_0, \alpha,a),x_0+L;a) \geq \frac{ \partial_t w_\infty (0,x_0+L;\alpha)}{2}>0.$$
On the other hand, using \eqref{winfty-periodic}, for any $\epsilon >0$, the following holds true for any $a$ large enough
$$\widehat{u} (\tau (x_0, \alpha,a),x_0+L;a) \geq \alpha- \epsilon.$$
Then one gets
$$\widehat{u} (\delta+\tau(x_0,\alpha,a),x_0+L;a) \geq \alpha - \epsilon +  \frac{ \partial_t w_\infty (0,x_0+L;\alpha)}{2} \; \delta.$$
By choosing $\epsilon$ small enough, we conclude that $\widehat{u} (\delta + \tau (x_0, \alpha,a);x_0+L;a) > \alpha$, thus $$\delta > \tau (x_0+L, \alpha,a) - \tau (x_0, \alpha, a) = \tau (x_0, \alpha,a-L) - \tau (x_0, \alpha, a),$$
for any $a$ large enough. In particular, one gets that the sequence $\{\tau_k\}$ is bounded, which contradicts the existence of a subsequence going to $+\infty$.
This completes the proof of the result.
\end{proof}\\

In this subsection, we have proven that the limit $w_\infty$ is either a periodic positive stationary solution or a pulsating traveling wave. In the monostable case, there is no periodic positive stationary solution between $0$ and $p$, so that~$w_\infty$ is always a pulsating traveling wave connecting $0$ to $p$, which already gives part~$(i)$ of Theorem~\ref{heavip1}. Together with Assumption~\ref{assumption-p1'}, it is clear that $p$ is isolated, so that one can choose $\alpha$ close enough to $p(x_0)$, so that function $w_\infty$ is a pulsating traveling wave connecting some periodic stationary solution $p_1$ to $p$.

\begin{remark}\label{CVtaukspeed}
Note that it follows from the above proof as well as the uniqueness of the speed, that in the case where $w_\infty$ is a pulsating traveling wave, then the whole sequence $\tau_k$ converges to $\frac{L}{c}$ where $c$ is the speed of $w_\infty$. This will be used later in the paper.
\end{remark}

\subsection{Construction of the terrace of traveling fronts}\label{sec:discussp*}

We now aim to construct a terrace composed of pulsating fronts.
We will proceed by iteration to construct such a terrace. Let us first notice that, as mentioned above, by choosing $\alpha$ close enough to $p(x_0)$, one can find a wave $U_1 (t,x)= w_\infty (t,x;\alpha)$ connecting some periodic stationary solution $p_1 <p$ to $p$. This gives us the first step of our iterative argument which is related to the following claim:
\begin{lemma}\label{iteration-terrace}
Assume that for some $0<\alpha_k<p(x_0)$, function $$U_k (t,x) := w_\infty (t,x;\alpha_k)$$ is a pulsating traveling wave connecting $ p_k>0$ to $p_{k-1}>p_k$.

Then $p_k$ is isolated from below and stable from below with respect to \eqref{eqn1-per}. Furthermore, there exists some $\alpha_{k+1} < p_k (x_0)$ such that $$U_{k+1} (t,x) :=w_\infty (t,x;\alpha_{k+1})$$ is a pulsating traveling wave connecting some stationary periodic solution $p_{k+1} <p_k$ to~$p_k$.
\end{lemma}

\begin{remark}
Note that the iteration will clearly end if one obtains
$p_k \equiv 0$ at some step $k$.
\end{remark}
In order to prove the above lemma, we begin with some preliminary
claims:

\begin{claim}\label{walphapk}
The stationary solution $p_k$ is steeper than any other entire
solution and, moreover,
\[
w_\infty (t,x;p_k (x_0)) \equiv p_k (x).
\]
\end{claim}
\begin{proof}
Let $v$ be an entire solution of \eqref{eqn1} such that $0<v<p$
and that $v(t_1,x_1)=p_k (x_1)$ for some $(t_1,x_1) \in \R^2$.
We know from Lemma \ref{lem:omega-steep} that
$w_\infty (\cdot,\cdot;\alpha_k)$ is steeper than any other entire
solution between $0$ and $p$, thus for any $t'$ and $t$ in $\R$,
\[
Z \left[w_\infty (t',\cdot;\alpha_k) - v(t,\cdot) \right] \leq 1,
\]
\[
SGN \left[w_\infty (t',\cdot;\alpha_k) - v(t,\cdot) \right] \lhd
[+ \; -].
\]
Passing to the limit as $t'\rightarrow -\infty$, one gets, for all
$t \in \R$,
\[
Z \left[p_k(\cdot) - v(t,\cdot) \right] \leq 1,
\]
\[
SGN \left[p_k (\cdot) - v(t,\cdot) \right] \lhd [+ \; -].
\]
This implies, by Corollary \ref{cor:steeper}, that $p_k$ is steeper
than $v$. Since $v$ is arbitrary, in particular, $p_k$ is steeper
than $w_\infty (\cdot,\cdot;p_k (x_0))$. On the other hand, we know by
Lemma \ref{lem:omega-steep} that $w_\infty (\cdot,\cdot;p_k (x_0))$
is steeper than $p_k$. Thus these two functions are steeper than
each other.  Furthermore, neither lies strictly above or below the
other since $p_k (x_0) = w_\infty (0,x_0;p_k(x_0))$. Thus we
conclude that $w_\infty (t,x;p_k(x_0)) \equiv p_k(x)$, which
completes the proof of Claim \ref{walphapk}.
\end{proof}

\begin{claim}\label{supersolpk}
Let $v\equiv v(t,x)$ be a given function satisfying $0< v (t,x) < p(t,x)$ and let $x(t)$ be a nondecreasing function moving with average speed $0<c<c_*$ (where $c_*$ the minimal speed of spreading of $\widehat{u}$ provided by Lemma~\ref{spread}). Assume that
$$v(t,x(t))=p_k (x(t)), \; \ \forall t \in \R,$$
$$v(t,x) \mbox{ is a super-solution of \eqref{eqn1} on } D:=\left\{ (t,x) | \; x \geq x(t) \right\}.$$
Then there exists a sequence $t_j \rightarrow +\infty$ such that for any $x \geq 0$, $$\liminf_{j \rightarrow +\infty} v (t_j,x(t_j)+ x) - p_k (x(t_j)+x) \geq 0.$$
\end{claim}
\begin{proof}
Let us look at the intersection of $\widehat{u} (t,\cdot;a)$ and $v_j (t,x) := v (t,x-jL)$ for any $j \in \mathbb{N}$. Note that $v_j$ is a super-solution for \eqref{eqn1} on the domain
$$D_j := \left\{ (t,x) \; | \ (t,x-jL) \in D \right\},$$
and, moreover, that $\widehat{u} (0,\cdot;a) = 0 \leq v_j (0,\cdot)$ on the half-space $x \geq x(0)+jL$ for each $j \in \mathbb{N}$ large enough.

Since $x(t)$ moves with the average speed $c$ smaller than the minimal spreading speed $c_*$ of $\widehat{u}$, one has that for any $j \in \mathbb{N}$, $$\widehat{u} (t,x(t)+jL;a) \rightarrow p(x(t)+jL) > v_j (t,x(t)+jL)=v (t,x(t)),$$
as $t \rightarrow +\infty$. Thus, for any $j$ large enough, there exists some minimal time $t_j$ such that $$\widehat{u} (t_j,x(t_j)+jL;a) = v_j (t_j,x(t_j)+jL)=p_k(x(t_j)).$$
Since $v_j$ is a super-solution on the domain $D_j$, one can check that
\begin{eqnarray*}
\widehat{u} (t_j,x;a) & \leq &  v_j (t_j,x) \mbox{ for any } x \geq x(t_j)+jL,\vspace{5pt} \\
& \leq & v (t_j,x-jL) \mbox{ for any } x \geq x(t_j)+jL.
\end{eqnarray*}
One can easily check that $t_j \rightarrow +\infty$ as $j \rightarrow+\infty$. Therefore, by standard parabolic estimates and possibly up a subsequence, one may assume that $x (t_j) \rightarrow x_\infty$ in $\R_{/ L\R}$ and that $\widehat{u} (t+t_j,x(t_j)-x_\infty +jL+x;a)$ converges as $j \rightarrow +\infty$ to some $\omega$-limit~$v_\infty$ that satisfies
$$v_\infty (0, x_\infty ) = p_k (x_\infty) .$$
Since $v_\infty$ is steeper than $p_k$, and conversely from Claim~\ref{walphapk}, it follows that $v_\infty \equiv p_k$. Therefore, we get that
$$\liminf_{j \rightarrow +\infty} v (t_j,x(t_j)+x) - p_k(x(t_j)+x) \geq 0,$$
for any $x \geq 0$ and the result follows.
 \end{proof}

We are now able to prove Lemma~\ref{iteration-terrace}.\\
\begin{proof}[Proof of Lemma~\ref{iteration-terrace}] We will split the proof of this lemma into several parts.

\paragraph{Step 1: $p_k$ is isolated from below.}
Assume by contradiction that there exists some sequence $\{q_j\}_j$ of periodic stationary solutions such that $q_j \rightarrow p_k$ as $j \rightarrow +\infty$ and $q_j < p_k$ for any $j \in \mathbb{N}$. Using standard elliptic estimates, one can easily show that the convergence holds uniformly in $C^1 (\mathbb{R})$.

Let us introduce the following principal eigenvalue problem:
\begin{equation}\label{principal-eigen}
\left\{
\begin{array}{rcl}
-\partial_{xx} \phi_\lambda + 2 \lambda \partial_x \phi_\lambda - \frac{\partial f}{\partial u} (x,p_k (x)) \phi_\lambda & = & \mu (\lambda) \phi_\lambda \ \mbox{ in } \R, \vspace{5pt}\\
\phi_\lambda >0 \mbox{ and  $L$-periodic}.
\end{array}
\right.
\end{equation}
Note that the principal eigenvalue $\mu (\lambda)$, which is associated to the linearized problem around $p_k$, satisfies:
\begin{itemize}
\item [$(i)$] $\mu (\lambda) - \mu (0) = O (\lambda^2)$ on a neighborhood of $0$;
\item [$(ii)$] $\mu (0) =0$.
\end{itemize}
The first statement $(i)$ follows from the following formula taken from~\cite{LLM10} (see Proposition~7.1), which is adapted from Nadin in~\cite{Nadin10}:
\begin{equation}\label{eq:Nadin_}
\mu (\lambda) = \min_{\overset{\eta \in H_{per}^1}{\;\eta>0}}\frac{1}{\int_0^L \eta^2 dx} \left(\mathcal F(p_k,\eta) 
+\lambda^2 \left( \int_0^L \eta^2 dx - \frac{L^2}{\int_0^L \eta^{-2}dx}\right) \right),
\end{equation}
where $\mathcal F(p_k,\eta)$ is the functional defined by
\begin{equation*}
\mathcal F(p_k,\eta)=\int_0^L \left( \eta'^2 - \frac{\partial f}{\partial u} (x,p_k (x)) \right) dx.
\end{equation*}
Next $(ii)$ follows from the fact that $p_k$ is an accumulation point of periodic stationary solutions.

Let us now construct some super-solution of \eqref{eqn1}. Consider the function $v$ defined by
$$v (t,x) :=  \min \left\{ p_k (x), e^{-\lambda ( x-ct)}  \phi_\lambda + q_j (x)\right\},$$
wherein $0 < c < c_*$ ($c_*$ being the minimal speed of spreading of $\widehat{u}$, provided by Lemma~\ref{spread}), $\lambda >0$ while $\phi_\lambda$ is a solution of \eqref{principal-eigen}.

It is clear that there exists some increasing map $t \mapsto x(t)$ moving with the average speed $c$ and such that
$$v(t,x(t))=p_k(x(t)), \; \ \forall t \in \R,$$
$$v(t,x) < p_k (x), \forall t \in \R \mbox{ and } \forall x > x(t).$$
Let us now define
$$D:= \left\{ (t,x) \; | \ x \geq x(t) \right\},$$
and compute on this set the following quantity:
\begin{equation*}
\begin{split}
\partial_t v - &\partial_{xx} v - f(x,v)  \\
=& e^{-\lambda (x-ct)} \left( (c\lambda - \lambda^2) \phi_\lambda + 2 \lambda \partial_x \phi_\lambda - \partial_{xx} \phi_\lambda \right) 
-\partial_{xx} q_j\\
& - f \left(x,  q_j + e^{-\lambda (x-ct)} \phi_\lambda \right) \\
 =&    e^{-\lambda (x-ct)} \left( (c\lambda - \lambda^2) \phi_\lambda + 2 \lambda \partial_x \phi_\lambda - \partial_{xx} \phi_\lambda \right) \\
&  - \frac{\partial f}{\partial u} (x, q_j) e^{-\lambda (x-ct)} \phi_\lambda + o\left(\min\left\{ p_k -q_j,  \phi_\lambda e^{-\lambda (x-ct)}\right\} \right)\\
 =&   e^{-\lambda (x-ct)} \left( (c\lambda - \lambda^2) \phi_\lambda + 2 \lambda \partial_x \phi_\lambda - \partial_{xx} \phi_\lambda \right) \\
&  - \frac{\partial f}{\partial u} (x, p_*) e^{-\lambda (x-ct)} \phi_\lambda + o\left(\min\left\{ p_k -q_j,  \phi_\lambda e^{-\lambda (x-ct)}\right\} \right)\\
 =&    e^{-\lambda (x-ct)}  (c\lambda - \lambda^2 - \mu (\lambda) ) \phi_\lambda+o\left(\min\left\{ p_k -q_j,  \phi_\lambda e^{-\lambda (x-ct)}\right\} \right),\\
 >&  0,
\end{split}
\end{equation*}
where the last inequality holds for any $j$ large enough and any $\lambda$ small enough, on the domain $D$.

Next Claim~\ref{supersolpk} applies and provides the existence of a sequence $t_j \rightarrow +\infty$ such that
\begin{equation}\label{claim+++}
\liminf_{j \rightarrow +\infty} v(t_j ,x(t_j)+x) - p_k (x(t_j)+x) \geq 0,
\end{equation}
as $j \rightarrow +\infty$ and for any $x \geq 0$. On the other hand, from the definition of $v$, there exists some constant $A>0$ such that for any $t \in \R$:
$$v (t,x(t)+A;a) < p_k (x(t)+A).$$
This contradicts \eqref{claim+++} and we conclude that $p_k$ is isolated from below with respect to \eqref{eqn1-per}.

\paragraph{Step 2: Stability from below}
To prove this statement we will argue by contradiction and we assume that $p_k$ is unstable from below with respect to \eqref{eqn1-per}. Let us distinguish two cases.

Assume first that $p_k$ is linearly unstable, that is $\mu (0) <0$ where $\mu$ is defined as in \eqref{principal-eigen}. Then, proceeding as in the well-known monostable case (see for instance \cite{BHR05}), one can find a stationary super-solution of the form $v(x):=p_k(x) - \kappa \psi_R (x)$, with $\kappa>0$ is small enough and wherein $\psi_R$ is a principal eigenfunction of the following problem:
\begin{equation}\label{principal-eigen_K}
\left\{
\begin{array}{l}
-\partial_{xx} \psi_R - \frac{\partial f}{\partial u} (x,p_k (x)) \psi_R  = \mu_R \psi_R \ \mbox{ in } (-R,R), \vspace{5pt}\\
\psi_R >0 \mbox{ and } \psi_R\left(\pm R\right) = 0.
\end{array}
\right.
\end{equation}
One can check, using the regularity of $f$ and the fact that $\mu_R \rightarrow \mu (0) <0$ as $ R\rightarrow +\infty$, that $v$ is a super-solution of \eqref{eqn1}. As before, one can then apply Claim~\ref{supersolpk} to reach a contradiction in this case.

Assume now that $p_k$ is not linearly unstable, namely $\mu (0) =0$. Since $p_k$ is unstable from below with respect to \eqref{eqn1-per}, we also know that there exists some entire solution $U(t,x)$, decreasing in time and periodic with respect to the space variable, that belongs to the unstable set of $p_k$ in the downward direction (we refer to \cite{Matano84}), that is, such that
$$U (t,x) < p_k (x) \ \forall (t,x) \in \R^2,$$
$$U(t,x) \rightarrow p_k (x) \; \mbox{ as } t\rightarrow -\infty.$$
Let $\{t_j\}_{j \in \mathbb{N}}$ be a sequence converging to $-\infty$ as $j \rightarrow +\infty$, and such that the $L$-periodic function
$$r_j (x) := U (t_j,x),$$ satisfies
$$\partial_{xx} r_j + f(x,r_j (x)) < 0.$$
Similarly as above, we construct a super-solution crossing $p_k$ and use Claim~\ref{supersolpk} to reach a contradiction. Let us consider the function
$$w(t,x) := r_j (x) + e^{-\lambda (x-ct)} \phi_\lambda (x),$$
wherein $0<c<c_*$ and $c_*$ is the minimal speed of spreading of $\widehat{u}$), $\lambda >0$, while $\phi_\lambda$ is a solution of \eqref{principal-eigen}. As before, one can check that function $w$ satisfies the hypotheses of Claim~\ref{supersolpk} and stays away from below to $p_k$ as $t \rightarrow +\infty$. We again reach a contradiction.

\paragraph{Step 3: Convergence to a pulsating traveling wave for $\alpha < p_k (x_0)$}
This last step is rather easier. We already know that for any $0 < \alpha < p_k (x_0)$, we have that $w_\infty (t,x;\alpha)$ is either a pulsating traveling wave or a positive and periodic stationary solution. But we have just shown that $p_k$ is isolated from below, therefore similarly as we have done before in the case $p_k=p$ to begin our iteration, for any $\alpha$ close enough to $p_k (x_0)$, $w_\infty (t,x;\alpha)$ is a pulsating traveling wave connecting some stationary solution $p_{k+1}$ to $p_k$.
\end{proof}\\

To conclude the proof of the existence of a propagating terrace, it remains to show that the sequence is finite (or equivalently, that $p_k \equiv 0$ at some step $k$).

Let us argue by contradiction and assume that it is not. Since the sequence $\{p_k\}_k$ is monotonically decreasing, it converges uniformly to some $p_\infty \geq 0$, a periodic stationary solution of \eqref{eqn1}. As in the proof of Lemma~\ref{iteration-terrace}, we use a super-solution crossing some $p_k$ to get a contradiction.

We introduce the following principal eigenvalue problem:
\begin{equation}\label{principal-eigen2}
\left\{
\begin{array}{rcl}
-\partial_{xx} \psi_\lambda + 2 \lambda \partial_x \psi_\lambda - \frac{\partial f}{\partial u} (x,p_\infty (x)) \psi_\lambda & = & \nu (\lambda) \psi_\lambda \ \mbox{ in } \R, \vspace{5pt}\\
\psi_\lambda >0 \mbox{ and  $L$-periodic}.
\end{array}
\right.
\end{equation}
The above defined eigen-problem is the same as \eqref{principal-eigen}, with $p_k$ is replaced by $p_\infty$. As before, one has that $\nu (\lambda)$ satisfies:
\begin{itemize}
\item [$(i)$] $\nu (\lambda) \geq \nu (0)$ for any $\lambda$, and $\nu (\lambda) - \nu (0) = O (\lambda^2)$ on a neighborhood of $0$;
\item [$(ii)$] $\nu (0) =0$.
\end{itemize}
Let us now introduce the following function
$$z (t,x) :=  e^{-\lambda ( x-ct)}  \psi_\lambda + p_\infty (x),$$
wherein $0 < c < c_* $, $\lambda >0$ and $\psi_\lambda$ is a solution of \eqref{principal-eigen2}. Using the same computations as in the proof of Lemma~\ref{iteration-terrace}, one gets that
\begin{equation*}
\partial_t z - \partial_{xx} z - f(x,z) > 0,
\end{equation*}
for all $\lambda$ small enough, on some domain of the form $\{ (t,x) | \; x \geq x(t) \}$, wherein $x(t)$ moves with the average speed $c$ and satisfies for some $k$ large enough and any $t \in \R$:
$$z (t,x(t))=p_k (x(t)) \mbox{ and } z (t,x) \leq p_k (x) \mbox{ for } x \geq x(t).$$
As $z (t,x) \rightarrow p_\infty (x) < p_k (x)$ as $x \rightarrow +\infty$ uniformly with respect to $t \in \R$, one can proceed as before to reach a contradiction together with Claim~\ref{supersolpk}.

We conclude that the iterative process stops in a finite number $N$ of steps. This allows us to construct a propagating terrace, which is called $T^*$.

\begin{remark}
In fact, we have not yet proven that the sequence $(c_k)_k$ of the speeds of the pulsating traveling waves $U_k$ is nondecreasing. However, this directly follows from Lemma~\ref{iteration-terrace}. Indeed Lemma~\ref{iteration-terrace} also states that we have a decreasing sequence $(\alpha_k)_k$ such that
$$U_k (t,x)= w_\infty (t,x;\alpha_k).$$
Then, from Remark~\ref{CVtaukspeed}, the speed $c_k$ of $U_k$ can be obtained as
$$c_k = \lim_{j \rightarrow +\infty} \frac{jL}{ \tau (x_0,\alpha_k,a-jL)},$$
and since $\alpha \mapsto \tau (x_0,\alpha,b)$ is increasing for any $x_0$ and $b$ (see Remark \ref{rem:period}), one obtains that the sequence $c_k$ is nondecreasing.
\end{remark}

To conclude the proof of Theorem~\ref{heavip2}, it remains to check that the propagating terrace $T^*$ satisfies all the required statements. This is in fact straightforward from all the above. Part~$(i)$ indeed immediately follows from the construction of the terrace and Lemma~\ref{iteration-terrace}. Part~$(ii)$ follows from Claim~\ref{walphapk} and the construction of the $U_k$ as some $\omega$-limit orbits of $\widehat{u}$.

Moreover, one can easily check that $T^*$ is minimal. Indeed, let us argue by contradiction by assuming that it is not. Then, one can easily check that there exists some $k$ and some traveling wave $V$ crossing $p_k$. This contradicts the fact that it is steeper than any other entire solution.

Lastly, let us check that any other minimal propagating terrace $T$ is equal to $T^*$. Let $T=((q_k)_k,(V_k)_k)$ be a given other minimal propagating terrace. Then it immediately follows from the definition that the two sequences $(q_k)_k$ and $(p_k)_k$ are identically equal. Then, for any $k$, $V_k$ and $U_k$ are steeper than each other (from Definition~\ref{def:terrace} and part~$(ii)$ of Theorem~\ref{heavip2}) and intersect, hence they are identically equal up to some time shift.

This ends the proof of Theorem~\ref{heavip2}.

\subsection{Locally or uniform convergence to the waves}

In this section, we prove the convergence part of Theorem~\ref{heavip1} as well as Theorem~\ref{heavip2-CV}.

Let us first show the locally uniform convergence to the pulsating traveling waves $(U_k)_{1\leq k \leq N}$ along the moving frames with speed $c_k$ and some sublinear drifts. Let us fix some $1 \leq k \leq N$. For any large enough $t$ , let us define $j(t) \in \mathbb{N}$ such that
$$j(t) \frac{L}{c_k} \leq t < \left(j(t) +1\right) \frac{L}{c_k},$$
and let us introduce
$$t_{j(t)} := \sum_{i=0}^{i=j(t)} \tau_i ,$$
wherein $\tau_j$ is defined as in Section~\ref{CVpulswave}, with $\alpha = \alpha_k$ chosen so that $U_k (\cdot, \cdot) = w_\infty (\cdot,\cdot;\alpha_k)$.
Let us now consider $m_k (t)$, the piecewized affine function, defined by
$$m_k (t)= t_{j(t)} -t \ \ \text{   if   } \ t =j(t)\frac{L}{c_k}.$$
Recall that the sequence $\{\frac{1}{j} \sum_{i=0}^{i=j} \tau_i\}_j$ converges to $\frac{L}{c_k}$, so that $m_k(t) = o(j(t))=o(t)$ as $t \to +\infty$.

Furthermore, since
$$U_k (t,x)= w_\infty (t,x;\alpha_k) = \lim_{j \rightarrow +\infty} \widehat{u} (t+\tau_0 + ... + \tau_j, x +jL;a)$$
where the above convergence is understood to hold locally uniformly with respect to $(t,x)\in\mathbb R^2$, and since $t+m_k (t)-t_{j(t)} \sim (t-j(t) \frac{L}{c_k})$ and $x+c_k t - j(t) L$ stay bounded, one can check that
$$\widehat{u} (t+m_k (t),x+c_k t;a) - U_k \left(t-j(t) \frac{L}{c_k},x - j(t)L +c_k t  \right) \to 0 \ \ \text{ as } \ t \to +\infty.$$
Thus, we obtain
$$\widehat{u} (t,x+c_k (t-m_k (t)) ;a) - U_k \left( t- m_k (t),x+c_k (t -m_k (t) )\right) \to 0 \ \ \text{ as } \ t \to +\infty, $$
wherein both of the above convergences hold locally uniformly with respect to~$x \in \R$. This completes, in the general case, the convergence result \eqref{eqn:CVcor1} stated in Theorem~\ref{heavip2-CV}.\\

It now remains to consider what happens "outside" of the moving frames with speed $(c_k)_{1 \leq k \leq N}$. This will follow from the following monotonicity property:

\begin{claim}\label{claim:period-mon}
For all $(t,x)\in [0, +\infty) \times \R$, one has
$$\widehat{u} (t,x;a) \geq \widehat{u} (t,x+L;a).$$
\end{claim}
\begin{proof}
This Claim directly follows from Remark~\ref{rem:period}.\end{proof}\\

Let us first look on the left of the terrace, that is, when 
$x + c_1 (t-m_1 (t)) \rightarrow -\infty$. In that case, we will 
use the fact that
\begin{equation}\label{eqn:asymp0}
\lim_{t \rightarrow +\infty} U_1 (t,x) \equiv p (x).
\end{equation}
Let $\delta >0$ be a given small enough number. From the asymptotics of $U_1$, there exists~$x_\delta$ such that for all $t$:
$$ p(x)-\frac{\delta}{2} \leq U_1(t-m_1 (t),x+c_1 (t-m_1 (t))) \leq p(x) \text{  for all } x\leq -x_\delta +L.$$
Next for each time large enough one has  for all $x\in [-x_\delta , x_\delta]$:
$$| \widehat{u} (t,x+c_1 (t-m_1 (t));a) - U_1 (t-m_1 (t),x+c_1 (t-m_1 (t)))| \leq \frac{\delta}{2}.$$
Then, using Claim~\ref{claim:period-mon}, one gets for all large enough $t$:
$$ p(x)-\delta \leq \widehat{u} (t,x+c_1 (t-m_1 (t));a) \leq p(x) \text{  for all } x\leq -x_\delta +L.$$

One can proceed similarly to get that for any $\delta >0$, there exists $C$ such that for any $x \geq C$,
$$ | \widehat{u} (t,x+c_N (t-m_N (t));a)  | \leq \delta. \vspace{3pt}$$
Note that under Assumption~\ref{assumption-p2'}, since $N=1$, the uniform convergence~\eqref{eqn:CVcor0} immediately follows from the above computations.\\

Lastly, let $1 \leq k \leq N$ be a given integer. Then one has
$$
\lim_{t \rightarrow -\infty} U_k (t,x) \equiv p_k (x) \mbox{  and  }
\lim_{t \rightarrow +\infty} U_{k+1} (t,x) \equiv p_k (x) .
$$
As above, one can use Claim~\ref{claim:period-mon} to show that there exists some constant $C>0$ such that for each large time $t$:
$$ p_k(x)+\delta \geq \widehat{u} (t,x+c_k (t-m_k (t));a)\text{  for all } x\geq C.\vspace{3pt}$$
$$ p_k (x)-\delta \leq \widehat{u} (t,x+c_{k+1} (t-m_{k+1} (t));a) \leq p(x) \text{  for all } x\leq -C.$$
This ends the proof of Theorem~\ref{heavip2-CV}.



\begin{thebibliography}{99}


\bibitem{An88} S.B. Angenent, The zero set of a solution of a parabolic equation, J. Reine. Angew. Math. 390 (1988), 79-96.


\bibitem{BMR08} M. Bages, P. Martinez. J-M. Roquejoffre, Large-time dynamics for a class of KPP type equations in periodic media, C.R. Math. Acad. Sci. Paris 346 (2008), 19-20.

\bibitem{BMR11} M. Bages, P. Martinez. J-M. Roquejoffre, How traveling waves attract the solutions of KPP-type equations, Trans. Amer. Math. Soc., to appear.

\bibitem{Ber02-per} H. Berestycki, F. Hamel, Front propagation in periodic excitable media, Comm. Pure Appli. Math. 55 (2002), 949-1032.


\bibitem{BHR05} H. Berestycki, F. Hamel, L. Roques, Analysis of the periodically fragmented environment model : I - Species persistence, J. Math. Biol. 51 (2005), 75-113.

\bibitem{BN92}H. Berestycki, L. Nirenberg, Traveling wave in cylinders, Annales de l'IHP, Analyse non lin{\'e}aire 9 (1992), 497-572.


\bibitem{Bramson83} M. Bramson, The convergence of solutions of the Kolmogorov nonlinear diffusion equation to travelling waves, Amer. Math. Soc. 44 (1983), No. 285.

\bibitem{DuMatano10} Y. Du, H. Matano, Convergence and sharp thresholds for propagation in nonlinear diffusion problems, J. Eur. Math. Soc. 12 (2010), 279-312.

\bibitem{DGM2} A. Ducrot, T. Giletti, H. Matano, Convergence to a critical pulsating wave given a fast decaying initial data, in preparation.

\bibitem{FZ11} J. Fang, X.-Q. Zhao, Bistable Traveling Waves for Monotone Semiflows with Applications, submitted.

\bibitem{FifMcL77} P.C. Fife, J. McLeod, The approach of solutions of nonlinear diffusion equations to traveling front solutions, Arch. Rational Mech. Anal. 65 (1977), 335-361.


\bibitem{Hamel08} F. Hamel, Qualitative properties of monostable pulsating fronts: exponential decay and monotonicity, J. Math. Pures Appl. 89 (2008), 355-399.

\bibitem{HNRR} F. Hamel, J. Nolen, J-M. Roquejoffre, L. Ryzhik, The $\frac{3}{2} \log t$ delay in KPP equations, preprint.

\bibitem{HamRoq11} F. Hamel, L.Roques, Uniqueness and stability properties of monostable pulsating fronts, J. Eur. Math. Soc. 13 (2011), 345-390.



\bibitem{Lau85} K. Lau, On the nonlinear diffusion equation of Kolmogorov, Petrovsky, and Piscounov, J. Diff. Eq. 59 (1985), 44-70.

\bibitem{LLM10} X. Liang, X. Lin, H. Matano, Maximizing the spreading speed of KPP fronts in two-dimensional stratified media, preprint (2010).

\bibitem{MR??} P. Martinez, J-M. Roquejoffre, Convergence to critical waves in KPP-type equations, preprint.

\bibitem{Matano78} H. Matano, Convergence of solutions of one-dimensional semilinear parabolic equations, J. Math. Kyoto Univ. 18 (1978), 221-227.

\bibitem{Matano84} H. Matano, Existence of nontrivial unstable sets for equilibriums of strongly order-preserving systems, J. Fac. Sci. Univ. Tokyo Sect. IA Math. 30 (1984), 645-673.

\bibitem{Nadin10} G. Nadin, The effect of the Schwarz rearrangement on the periodic principal eigenvalue of a nonsymmetric operator, SIAM J. on Math. Anal. 41 (2010), 2388-2406.

\bibitem{Nadin??} G. Nadin, personal communications.

\bibitem{Uchi78} K. Uchiyama, The behavior of solutions of some non-linear diffusion equations for large time, J. Math. Kyoto Univ. 18 (1978), 543-508.

\bibitem{Wein02} H. Weinberger, On spreading speed and travelling waves for growth and migration models in a periodic habitat, J. Math. Biol. 45 (2002), 511-548.

\end{thebibliography}
\end{document}